\documentclass[12pt,reqno]{amsart}
\usepackage{latexsym,amsmath,amssymb,mathrsfs,setspace,pstricks,amsthm}

\newtheorem{theo}{Theorem}
\newtheorem{lemma}[theo]{Lemma}
\newtheorem{corollary}[theo]{Corollary}

\theoremstyle{remark}

\begin{document}
\doublespace

\title{ Units in $FD_{2p^m}$}

\author{ Kuldeep Kaur, Manju Khan}
\email{  kuldeepk@iitrpr.ac.in, manju@iitrpr.ac.in}

\address{Department of Mathematics, Indian Institute of Technology
Ropar,Nangal Road, Rupnagar - 140 001 }

\footnotetext{}


\maketitle

\markboth{ Kuldeep Kaur, Manju Khan}{Units in $FD_{2p^m}$}

\doublespacing

\begin{abstract}
In this note, we compute the order and provide the structure of the unit group
$\mathcal{U}(FD_{2p^m})$ of the group algebra $FD_{2p^m}$, where $F$ is a finite
field of characteristic $2$ and $D_{2p^m}$ is the dihedral group of order
$2p^m$ such that $p$ is an odd prime. Further, we obtain the structure of the
unitary subgroup
$\mathcal{U}_*(FD_{2p^m})$ with respect to canonical involution $*$ and prove
that it is a normal subgroup of the unit group $\mathcal{U}(FD_{2p^m})$.
\end{abstract}
\section*{introduction}
 \noindent Let $FG$ be the group algebra of the group $G$ over the field $F$ and
 $\mathcal{U}(FG)$ its unit group. Extending the anti-automorphism $g
\mapsto
 g^{-1}$ of $G$ to the group algebra, we obtain the
involution \[x = 
 \sum_{g \in G} x_g g \mapsto x^* = \sum_{g \in G} x_g g^{-1}\] of $FG$, which
we
call canonical involution and is denoted by $*$. An element $x \in
\mathcal{U}(FG)$ is called unitary if $x^* = x^{-1}$. Let $\mathcal{U}_*(FG)$ be
the unitary subgroup consisting of all unitary elements of $\mathcal{U}(FG)$. If
$ \displaystyle V(FG) = \{ \sum_{g \in G}a_gg \in \mathcal{U}(FG) \ | \
\sum_{g\in G} a_g = 1\}$ is the normalized unit group of the group algebra $FG$,
then $\mathcal{U}(FG) = V(FG) \times F^*$, where $F^*$ is the set of all non
zero elements of $F$. In particular, if $F$ is a finite field of characteristic
2, then the unitary subgroup $\mathcal{U}_*(FG)$ is same as the unitary subgroup
$V_*(FG)$. 
Interest in the group $\mathcal{U}_*(FG)$ arose in algebraic topology and
unitary $K$-theory \cite{MR0292913}.

 Sandling \cite{MR761637} computed a basis for the normalized unit group
$V(FG)$, if $G$ is a finite abelian $p$-group and $F$ is a finite field with $p$
elements.
 In \cite{MR1136226} Sandling provided the generators and relators for each
2-group of order dividing 16 over finite field with 2 elements.  Creedon and
Gildea in
\cite{MR2884238} obtained the structure of $V(FD_8)$, where $D_8$
is a dihedral group of order 8 and $F$ is a finite field of characteristic 2. In
\cite{KM}, Kuldeep and Manju described the structure
of the
unit group
$\mathcal{U}(F_2D_{2p})$, where $F_2$ is the finite field with two elements and
$D_{2p}$ is the dihedral group of order $2p$ for an odd prime $p$. They also
computed the order and provided the structure of the unitary subgroup
$\mathcal{U}_*(F_2D_{2p})$ with respect to canonical involution $*$.

Our goal here is to study the structure of the unit group
$\mathcal{U}(FD_{2p^m})$ of the group algebra $FD_{2p^m}$ of the dihedral
group $D_{2p^m}$ of order $2p^m$, where $p$ is an odd prime over a finite field
$F$
of characteristic 2. We also provide the structure of the unitary subgroup
$\mathcal{U}_*(FD_{2p^m})$. 

Throughout the paper, we assume that $F$ is a field with $2^n$ elements. 
The polynomial $x^{p^m}-1$ can be written as
\[x^{p^m} - 1 = (x - 1)\prod_{s = 1}^m \phi_{p^s}(x),\] where $\phi_{p^s}(x)$ is
the $p^s$-th cyclotomic polynomial over $F$. If $o_s$
denotes the order of $2^n$ modulo $p^s$, i.e., $o_s = o_{p^s}(2^n)$ for $1 \leq
s \leq m$, 
and $k_s$ denotes the number of 
irreducible factors of $\phi_{p^s}(x)$ over $F$, then $k_s=
\frac{\phi(p^s)}{o_s}$, where
$\phi$ denotes the Euler $\phi$- function. If $\alpha$ is a primitive
$p^s$-th root of unity, then the polynomail $(x - \alpha)(x -
\alpha^{2^n} )\cdots (x - \alpha^{(2^n)^{o_s -1 }})$ is a minimal polynomial of
$\alpha$ over $F$ and therefore, it is an irreducible factor of $\phi_{p^s}(x)$
over $F$. 
\section*{Unit Group of $FD_{2p^m}$}
In this section, we provide the order of the unit group
$\mathcal{U}(FD_{2p^m})$.
Assume that $o_p(2^n) = d$. Clearly, $p$
and $d$ are relatively prime and $o_{p^2}(2^n)$ is either $d$ or $pd$.  
In general 
\begin{lemma}
For an integer $i \geq 1$,  if $o_{p^i}(2^n)$ is $d$, then $o_{p^{i+1}}(2^n)$ is
either $d$ or $pd$. 
\end{lemma}
\begin{proof}
Suppose that $o_{p^{i+ 1}}(2^n)$ is $t$. Then, $d |t$.
Further, since $(2^n)^d = 1 \ mod \ p^i$, it follows that  $(2^n)^{dp} = 1 \mod 
\ p^{i+ 1}$.
Therefore, $t|pd$. Hence, either $t = d$ or $t = pd$. 
\end{proof}

\begin{lemma}
 Assume that $o_{p^i}(2^n)$ is $d$. If $o_{p^{i+1}}(2^n)$ is $pd$, then
$o_{p^{k}}(2^n)$ is
$p^{k-i}d$ for all $k \geq i + 2$. 
\end{lemma}

 \begin{proof}
  Let $o_{p^{i+2}}(2^n)$ be $t$. Then, it is clear that $pd|
t$ and
$d|t$. Further, since $(2^n)^d \equiv 1 \ \mod \ p^i$, there is an integer $s$
such that $(2^n)^d = 1 + sp^i$ and so $ (2^n)^{dp} \equiv  1 + sp^{i+ 1} \ \mod
\ p^{i + 2}$.  
As $o_{p^{i+1}}(2^n)$ is $pd$, it follows that $(p, s) = 1$
and so
$(2^n)^{dp} \not \equiv 1 \mod \ p^{i + 2} $. 
Also $(2^n)^{pd} \equiv 1$ modulo $p^{i+1}$, which entails that $(2^n)^{p^2d}
\equiv 1$ modulo $p^{i+2}$. Therefore $t|p^2d$. Hence $o_{p^{i+2}}(2^n)$ is
$p^2d$. In a similar way one can prove for any integer $k \geq (i + 3)$.
\end{proof}

\begin{lemma}
Let $o_{p^k}(2^n)$ be $d$, for $1
\leq k \leq i$ and $o_{p^{i + 1}}(2^n)$ be $pd$. If $\zeta$ is a
primitive $p^m$-th root of unity, then $\zeta^j$ and
$\zeta^{-j}$ are roots of the same irreducible polynomial over $F$ if and
only if
$2|d$.
\end{lemma} 
\begin{proof}
First, assume that $\zeta^j$ and $\zeta^{-j}$ are roots of the same
irreducible polynomial over $F$. If $j=kp^r, (k,p)=1$, then it is clear that
$\zeta^{j}$ is primitive $p^{m-r}$-th root of unity. Therefore, $\zeta^j$ and
$\zeta^{-j}$ are roots of
$\phi_{p^{m-r}}(x)$. Hence,
$-1 \equiv 1(2^n)^t \mod p^{m-r} $, where $1 \leq t \leq o_{m-r}-1$.
  Thus $o_{p^{m - r}}(2^{nt}) $ is 2. Further, since $(2^n)^{o_{m-r}} \equiv 1
\mod p^{m-r}$,
it follows that $2|o_{m-r}$. From lemma (2), we get
 \[o_{m-r}= \begin{cases}
 d , & \textrm{if } m-r \leq i\\
 p^{m-r-i}d, & \textrm{if } m-r > i 
 \end{cases}\]
 As $(2, p)=1$, we obtain that $d$ is even.  

To prove the converse part, it is enough to show that if $d$ is even then there
exists an integer $t$ such that $-1 \equiv (2^n)^t \ mod \ p^{m- r}$ for any
$j=kp^r, (k,p)=1$. For
that,
if $d=2t$, then $(2^n)^{2t}=1 \mod p^{m-r}$ or
$(2^n)^{p^{m-r-i}2t}=1 \mod p^{m-r}$. Since $-1$ is the only element of order
2 in $\mathbb{Z}_{p^{m-r}}$, it follows that $(2^n)^t \equiv -1 \mod p^{m-r}$ or
$(2^n)^{p^{m-r-i}t} \equiv -1 \mod p^{m-r}$.
\end{proof}

\begin{lemma}
 Let $o_{p^k}(2^n)=d$ for $1 \leq k \leq i$ and $o_{p^{i+1}}(2^n)=pd$. If d is
even, then for any $j, [F(\zeta^{p^j}+\zeta^{-p^j}):F]=
\frac{o_{m-j}}{2}$, and if d is odd, then
$[F(\zeta^{p^j}+\zeta^{-p^j}):F]=
o_{m-j}$ and in this case $F(\zeta^{p^j}+\zeta^{-p^j})=F(\zeta^{p^j})$.
\end{lemma}
\begin{proof}
 If $[F(\zeta^{p^j}+\zeta^{-p^j}):F]= s < \frac{o_{m-j}}{2}$, then
there exists a polynomial of degree $2s < o_{m-j}$ satisfied by
$\zeta^{p^j}$, which is a contradiction. Therefore 
$[F(\zeta^{p^j}+\zeta^{-p^j}):F] \geq \frac{o_{m-j}}{2}$. 
Since $[F(\zeta^{p^j}):F]= o_{m-j}$, it follows that 
$[F(\zeta^{p^j}+\zeta^{-p^j}):F] = \frac{o_{m-j}}{2} $ or
$o_{m-j}$. Now, if d is even, then, by lemma (3), there is a polynomial
of degree $o_{m-j}-1$ over F satisfied by $\zeta^{p^j}+ \zeta^{-p^j}$.
Hence, $[F(\zeta^{p^j}+\zeta^{-p^j}):F] = \frac{o_{m-j}}{2}  $. Now if d
is odd, from lemma (3), it is clear that $[F(\zeta^{p^j}+\zeta^{-p^j}):F]
=o_{m-j}$ and hence $F(\zeta^{p^j}+\zeta^{-p^j})=F(\zeta^{p^j})$.
\end{proof}

  \begin{theo}
Let $D_{2p^m}$ be the dihedral group \[D_{2p^m} = \langle a, b \ | \ a^{p^m} =
1,
 b^2 = 1, b^{-1} a b = a^{-1}\rangle.\] Assume that $o_{p^k}(2^n)=d$, for $1
\leq k \leq i$ and $o_{p^{i+1}}(2^n)=pd$. 
If $o_r= o_{p^r}(2^n)$, then the order of the unit group
$\mathcal{U}(FD_{2p^m})$ is 
\[ 2^n(2^n-1)\prod_{r=1}^{m}(({q_r}^2-1)({q_r}^2-q_r))^{k_r'},\]
where $q_r= \begin{cases}
 2^{\frac{no_r}{2}}, & \text{ if } d \text{ is even}\\
  2^{no_r}, & \text{ if } d \text{ is odd}
  \end{cases}$ \hspace{1cm} 
  and $k_r'= \begin{cases}
 k_r, & \text{ if } d \text{ is even}\\
  \frac{k_r}{2}, & \text{ if } d \text{ is odd.}
  \end{cases}$
 \end{theo}
\begin{proof}
Let $f_{1,p^j}(x), f_{2,p^j}(x), \ldots, f_{k_j,p^j}(x)$ be the distinct
irreducible
factors of the cyclotomic polynomial $\phi_{p^j}(x)$ over F, where $k_j=
\frac{\phi(p^j)}{d}= \frac{p^{j-1}(p-1)}{d}$, for $1 \leq j \leq
i-1$ and $k_j=\frac{p^{i-1}(p-1)}{d}$ for $i \leq j
\leq m$. 
If $\gamma_1$ and $\gamma_2$ denote roots of distinct irreducible factors of
$\phi_{p^j}(x)$ over $F$, then $\gamma_1=\zeta^{p^{m-j}J_1}$
and $\gamma_2=\zeta^{p^{m-j}J_2}$ for some integer $J_1$ and  $J_2$ such that
$(J_1,p)=1$ and $(J_2,p)=1$. Choose
$\gamma_1'=\zeta^{p^{m-j-1}J_1}$
and $\gamma_2'=\zeta^{p^{m-j-1}J_2}$, then $\gamma_1'$ and $\gamma_2'$ are roots
of $\phi_{p^{j+1}}(x)$ over $F$. If  $\gamma_1'$ and $\gamma_2'$ are roots of
the
same irreducible factors of $\phi_{p^{j+1}}(x)$ then there exists $t$ such
that $J_1= (2^n)^tJ_2 \mod p^{j+1}$. Hence,  $J_1= (2^n)^tJ_2 \mod p^{j}$, which
is a
contradiction. Hence, $\gamma_1'$ and $\gamma_2'$ are roots of distinct
irreducible factors of $\phi_{p^{j+1}}(x)$ over F. 

Choose $\gamma_{1j}$ as $a$ root of the irreducible factor $f_{j, p}(x)$ of
$\phi_p(x)$, where $1 \leq j \leq k_1$. 
Therefore $\gamma_{1j}=
\zeta^{p^{m-1}J_j}, (J_j, p)=1$. Let $\gamma_{11}^{(l)},
\gamma_{12}^{(l)}, \cdots , \gamma_{1k_1}^{(l)}$ denote roots of distinct
irreducible factors of $\phi_{p^l}(x)$ over F, for $1 \leq l \leq m$, where
$\gamma_{1j}^{(1)}= \gamma_{1j}$ and $\gamma_{1j}^{(l)}=
\zeta^{p^{m-l}J_j}, 1 \leq j \leq k_1$. Since the number of irreducible
factors of $\phi_{p^2}(x)$ is $k_2= pk_1$, without loss of generality, assume
that $\gamma_{1j}^{(l)}$ is a root of $f_{j,p^l}(x)$ for $1 \leq j \leq k_1$ and
$1
\leq l \leq m$. Further, in a similar way, choose $\gamma_{2j}$, from the
irreducible factor $f_{j,p^2}$ of $\phi_{p^2}(x)$, where $k_1 < j \leq k_2$.

If $\gamma_{rs}$ denotes the root of the irreducible factor $f_{s,p^r}(x)$ of
$\phi_{p^r}(x)$ for $1 \leq r \leq i, 1 \leq s \leq k_r$. \\
Define $\gamma_{rs}= \begin{cases}
                      \gamma_{1s}^{(r)}, & \textrm{ if } 1 \leq s \leq k_1 \\
                      \gamma_{2s}^{(r-1)}, & \textrm{ if } k_1 < s \leq k_2 \\
                      \vdots \\
                      \gamma_{rs}^{(1)}, & \textrm{ if } k_{r-1} < s \leq k_r
,\\
                     \end{cases}$
                     
where $ \gamma_{rs}^{(1)}= \gamma_{rs}$.\\
Since $k_i=k_{i+1}= \cdots = k_m$, consider for $i+1 \leq r \leq m, 1 \leq s
\leq k_i $\\
$\gamma_{rs}= \begin{cases}
                      \gamma_{1s}^{(r)}, & \textrm{ if } 1 \leq s \leq k_1 \\
                      \gamma_{2s}^{(r-1)}, & \textrm{ if } k_1 < s \leq k_2 \\
                      \vdots \\
                      \gamma_{is}^{(r-i+1)}, & \textrm{ if } k_{i-1} < s \leq
k_i .\\
                     \end{cases}$\\
Define a matrix representation $T_{rs}$ of $D_{2p^m}$,  \[T_{rs} : D_{2p^m}
\rightarrow
M_2(F(\gamma_{rs}+ \gamma_{rs} ^{-1}))\] by the assignment

\[a \mapsto  \left(
\begin{array}{cc}
  0 & 1 \\
 1  & \gamma_{rs} + \gamma_{rs}^{-1} \\
\end{array}
\right), \ \ \  b \mapsto  \left(
\begin{array}{cc}
  1 & 0 \\
 \gamma_{rs} + \gamma_{rs}^{-1} & 1 \\
\end{array}
\right)\]
 If $d$ is even, then define $T = T_0 \displaystyle \mathop{\oplus}_{r=1}^{m}
\mathop{\oplus}_{s=1}^{k_r} T_{rs}$ , the direct sum of the given
representations
$T_{rs}, 1 \leq r \leq m, 1 \leq s \leq k_r$, and $T_0$ is the trivial
representation of
$D_{2p^m}$ over $F$ of degree 1. 

\noindent Suppose $d$ is odd. Then lemma (4) implies that $\gamma_{rs}$ and
$\gamma_{rs}^{-1}$ are roots of different irreducible factors of
$\phi_{p^r}(x)$. If $\gamma_{rs}^{-1}$ is a root of $f_{j,p^r}(x)$, then choose
$\gamma_{rj} =
\gamma_{rs}^{-1}$. Without loss of generality, assume that $\{ \gamma_{rs}| 1
\leq r \leq m, 1 \leq s \leq \frac{k_r}{2}\}$ are roots of distinct
irreducible factors of $\phi_{p^r}(x)$ such that $\gamma_{rj} \neq
\gamma_{rs}^{-1}$ for $1 \leq j, s \leq \frac{k_r}{2}$.
 Then, define $T = \displaystyle T_0 \mathop{\oplus}_{r =
1}^{m}\mathop{\oplus}_{s=1}^{\frac{k_r}{2}} T_{rs}$, the direct sum of all
distinct matrix representations. Therefore, the map  \[T : D_{2p^m} \rightarrow
F
\displaystyle \mathop \oplus_{r=1}^{m} \mathop {\oplus}_{s=1}^{k_r'}
M_2(F(\gamma_{rs}+ \gamma_{rs} ^{-1})\]
is a group homomorphism, where $k_r'=
\begin{cases}
 k_r, & \text{ if } d \text{ is even}\\
  \frac{k_r}{2}, & \text{ if } d \text{ is odd.}
  \end{cases}$\\
 By extending this group homomorphism linearly over $F$, we obtain an algebra
homomorphism 
 \[T': FD_{2p^m} \rightarrow  F
\displaystyle \mathop \oplus_{r=1}^{m} \mathop {\oplus}_{s=1}^{k_r'}
M_2(F(\gamma_{rs}+ \gamma_{rs} ^{-1})).\]
\noindent Consider the matrix representation $S_{rs}$ of $D_{2p^m}$
defined as follows: 
\[S_{rs}(a)= \left(
\begin{array}{cc}
  \gamma_{rs} & 0 \\
 0  & \gamma_{rs}^{-1} \\
\end{array}
\right),   \ \ S_{rs}(b)= \left(
\begin{array}{cc}
  0 & 1 \\
 1  & 0 \\
\end{array}
\right).\] 
Note that for any $x \in D_{2p^m}, T_{rs}(x) = M_{rs}
S_{rs}(x)M_{rs}^{-1}$,\\ where 
$M_{rs}=\left(
\begin{array}{cc}
  1 & 1 \\
  \gamma_{rs} &   \gamma_{rs}^{-1} \\
\end{array}
\right)$.\\
 Let $x = \displaystyle
\sum_{i=
0}^{p^m-1}\alpha_{i} a^i +  \sum_{i =
0}^{p^m-1}\beta_{i}a^ib \in Ker T'$.
 Then, the equation $T'(x)=0$ implies that
\begin{equation}
 \sum_{i=0}^{p^m-1}\alpha _{i}+
\sum_{i=0}^{p^m-1}\beta _{i}=0
\end{equation}
 and $\gamma_{rs}, \gamma_{rs}^{-1}, 1 \leq r \leq m, 1 \leq s \leq k_r'$ are
roots of the polynomials  $g(x)=\alpha_{0}+\alpha_{1}x+ \cdots +
\alpha_{p^{m}-1}x^{p^m-1}$ and
 $h(x)=\beta_{0}+\beta_{1}x+ \cdots + \beta_{p^m-1}x^{p^m-1}$ over $F$. 
It follows that  irreducible factors of $\phi_{p^{r}}(x)$ are factors of $g(x)$
and $h(x)$ for all $1 \leq r \leq m$. Further, since the irreducible factors
$\phi_{p^r}(x)$ are co-prime,
it follows that
$\phi_{p}(x)\phi_{p^2}(x) \cdots \phi_{p^m}(x)$ divides $g(x) \ \text{and} \ 
h(x)$, i.e., $1+x+x^2+ \cdots + x^{p^m-1}$ divides $g(x) \ \text{and} \  h(x)$,
and hence $\alpha_i =
\alpha_j, \text{and}  \ \beta_i =
\beta_j , 0 \leq i, j \leq {p^m-1}$.
Thus, from equation $(1)$, we have 
$\alpha_i = \beta_i, 0 \leq i \leq p^m-1$ and therefore $Ker T' =
F \widehat{D_{2p^m}}$. Since dimensions of $(FD_{2p^m}/F\widehat{D_{2p^m}})$
and
$F \displaystyle \mathop \oplus_{r=1}^{m} \mathop {\oplus}_{s=1}^{k_r'}
M_2(F(\gamma_{rs}+ \gamma_{rs} ^{-1}))$ over $F$ are equal, we obtain
\[\frac{FD_{2p^m}}{F \widehat{D_{2p^m}}} \cong F \mathop \oplus_{r=1}^{m}
\mathop {\oplus}_{s=1}^{k_r'}
M_2(F(\gamma_{rs}+ \gamma_{rs} ^{-1})).\] Note that
$F\widehat{D_{2p^m}}$
is a nilpotent ideal and hence $T'$ induces an epimorphism
$T'':\mathcal{U}(FD_{2p^m}) \rightarrow 
F^* \times \displaystyle \mathop \prod_{r=1}^{m} \mathop {\prod}_{s=1}^{k_r'}
GL_2(F(\gamma_{rs}+ \gamma_{rs} ^{-1}))$ such that $ker T'' =
 1+F\widehat{D_{2p^m}}$. Thus 
\[\frac{\mathcal{U}(FD_{2p^m})}{ 1+F\widehat{D_{2p^m}}} \cong
F^* \times \displaystyle \mathop \prod_{r=1}^{m} \mathop {\prod}_{s=1}^{k_r'}
GL_2(F(\gamma_{rs}+ \gamma_{rs} ^{-1})).\]and hence the result follows.

\end{proof}
\section*{Structure of The Unitary Subgroup $\mathcal{U}_*(FD_{2p^m})$}
In this section, we study the structure of the unitary subgroup
$\mathcal{U}_*(FD_{2p^m})$ with respect to canonical involution $*$.
We also prove that the unit group $\mathcal{U}(FD_{2p^m})$ is the direct product
of
the unitary subgroup with a central subgroup. 

 Let $f(x)$ be a monic irreducible polynomial of degree $n$
over $F_2$ such that $F \cong F_2[x]/ \langle f(x) \rangle$ and let
$\alpha$ be
the residue class of $x \  \textrm{modulo} \ \langle f(x)\rangle$.
Let us define the set $B$ as follows: \[B = \{1+\alpha^i(a^j+a^{-j})(1+a^kb) \ |
\ 0
\leq i\leq n-1, 1 \leq
j \leq
\frac{p^m-1}{2}, 0 \leq k \leq p^m-1\}.\] 
Clearly, $B \subseteq U_*(FD_{2p^m})$. Let $\mathcal{B}(FD_{2p^m})$
denote the 
group generated by $B$. The following theorem provides the structure of this
group.

\begin{theo}
Let $o_{p^k}(2^n)=d$, for $1
\leq k \leq i$ and $o_{p^{i+1}}(2^n)=pd$. Then, 
 $\mathcal{B}(FD_{2p^m}) \cong
\displaystyle \mathop \prod_{r=1}^{m} \mathop {\prod}_{s=1}^{k_r'}
SL_2(F(\gamma_{rs}+ \gamma_{rs} ^{-1}))$, where
  $SL_2(K)$ is the special linear group of degree 2 over the field $K$.
\end{theo}
 We will need the following results:
\begin{lemma}
$ D_{2p^m} \cap \mathcal{B}(FD_{2p^m})= \langle a \rangle$.
\end{lemma}

\begin{proof}
  Since $D_{2p^m}$ is in the normalizer of $\mathcal{B}(FD_{2p^m})$, it
implies that
 $\mathcal{B}(FD_{2p^m}) \cap D_{2p^m}$ is a normal subgroup of
$D_{2p^m}$.
 Therefore, it is either a trivial subgroup or $\langle a^{p^i} \rangle$, for $0
\leq i \leq m-1$. 
Let us define a map \[f : D_{2p^m} \rightarrow  \langle g \ | \ g^2 = 1\rangle\]
such
that
  $f(a^i) = 1 \ \textrm{and} \ f(a^i b) = g, 0 \leq i \leq p^m -1.$ Note that
 it is a group homomorphism and  we can extend this linearly to an algebra
homomorphism
 $f'$ from $FD_{2p^m}$ to $F \langle g \rangle$. It is easy to see
that the
image of the elements of $B$ under $f'$ is 1 and therefore $b \notin
\mathcal{B}(FD_{2p^m})$.
 Hence, $D_{2p^m} \cap \mathcal{B}(FD_{2p^m}) \neq D_{2p^m}$.
Further, observe that \[u_{ab, a}u_{ab, a^2} \cdots u_{ab,a^{\frac{p^m-1}{2}}}=
ab(1 + \widehat{D_{2p^m}})\]  \[\text{and}\ \ u_{b, a}u_{b, a^2} \cdots u_{b,
a^{\frac{p^m-1}{2}}}=b(1 +\widehat{D_{2p^m}}),\]
 where $u_{a^jb, a^i} = 1 + (a^i + a^{-i})(1 + a^jb)$. Consequently,
\[a=u_{ab, a}u_{ab, a^2} \cdots
u_{ab,a^{\frac{p^m-1}{2}}} u_{b, a^{\frac{p^m-1}{2}}} \cdots u_{b, a^2}u_{b,
a}.\] Hence, $ D_{2p^m} \cap \mathcal{B}(FD_{2p^m})= \langle a \rangle$.
\end{proof}

\begin{lemma}
If $\gamma_{rs}$ denotes the root of the irreducible factor $f_{s, p^r}$ of the
cyclotomic polynomial $\phi_{p^r}(x)$ over $F$, then the minimal polynomials of
$\gamma_{r_1s_1}+\gamma_{r_1s_1}^{-1}$ and
$\gamma_{r_2s_2}+\gamma_{r_2s_2}^{-1}$ are distinct.
\end{lemma}
\begin{proof}
First assume that $r_1 = r_2 = r$. Then, if $d$ is even, we have
$[F(\gamma_{rs_i} + \gamma_{rs_i}^{-1}) : F] = \frac{o_r}{2}$, where $1 \leq i
\leq 2$. Let  
$f(x) = a_0 + \cdots + a_{\frac{o_r}{2}}x^{\frac{o_r}{2}}$ be the minimal
polynomial over $F$ satisfied by both $\gamma_{rs_1} + \gamma_{rs_1}^{-1}$
and
 $\gamma_{rs_2}+ \gamma_{rs_2}^{-1}$. It follows that $\gamma_{rs_1}$ and
$\gamma_{rs_2}$ satisfy the same minimal polynomial over $F$. This is a
contradiction, because the irreducible factors $f_{s_1, p^r}$ and $f_{s_2, p^r}$
are coprime. 
Next, if $d$ is odd, then $[F(\gamma_{rs_i} + \gamma_{rs_i}^{-1}): F ] = o_r$.
Therefore, the degree of the minimal polynomial of $\gamma_{rs_i} +
\gamma_{rs_i}^{-1}$ over $F$ is $o_r$. Since
$\gamma_{rs_i}$ and
$\gamma_{rs_i^{-1}}$ are roots of different irreducible factors of
$\phi_{p^r}(x)$ with degree $o_r$ and the factors are co-prime, it implies that
the minimal polynomials of $\gamma_{rs_1}+\gamma_{rs_1}^{-1}$ and
$\gamma_{rs_2}+\gamma_{rs_2}^{-1}$ are distinct.

Now assume that $r_1 < r_2$. If $i < r_1 < r_2$, then the degree of the minimal
polynomial of $\gamma_{r_1s_1}+ \gamma_{r_1s_1}^{-1}$ is $\frac{p^{r_1-i}d}{2}$
and the degree of the minimal polynomial of $\gamma_{r_2s_2} +
\gamma_{r_2s_2}^{-1}$ is $\frac{p^{r_2-i}d}{2}$, if $d$ is even. Otherwise the
degrees are $p^{r_1-i}d$ and $p^{r_2-i}d$, respectively. Hence, their minimal
polynomials are distinct. 

Now, if $r_1 \leq i < r_2 $, then the degree of the minimal polynomial of 
 $\gamma_{r_1s_1}+\gamma_{r_1s_1}^{-1}$ is $\frac{d}{2}$ and the
degree of the 
 minimal polynomial of $\gamma_{r_2s_2}+\gamma_{r_2s_2}^{-1}$ is
$\frac{p^{r_2-i}d}{2}$, if $d$ is 
 even. Otherwise the degree of the minimal polynomials are $d$ and $p^{r_1-i}d$,
respectively. Hence, the result follows. In a similar way one can prove for
other
cases as well.   
\end{proof}
\textbf{Proof of the theorem}: Observe that the image of elements
of $B$ are in
$\displaystyle \mathop \prod_{r=1}^{m} \mathop {\prod}_{s=1}^{k_r'}
SL_2(F(\gamma_{rs}+ \gamma_{rs} ^{-1}))$ under the map $T''$. Suppose $T'''$ is
the restricted map of $T''$ to
$\mathcal{B}(FD_{2p^m})$,  i.e.,
\[T''':\mathcal{B}(FD_{2p^m}) \rightarrow 
\displaystyle \mathop \prod_{r=1}^{m} \mathop {\prod}_{s=1}^{k_r'}
SL_2(F(\gamma_{rs}+ \gamma_{rs} ^{-1}))\] such that
$T'''(x)=T''(x)$ for $ x \in \mathcal{B}(FD_{2p^m})$. Therefore, $kerT''' \leq
kerT'' = 1 + F\widehat{D_{2p^m}}$. Consider the element $t_\alpha =  1 + \alpha
\widehat{D_{2p^m}} $, where $\alpha$ is a non zero element of $F$. Now, if
$\alpha \neq 1$, then the length of $1 + \alpha \widehat{D_{2p^m}}$ is $2p^m$.
Since the length of each element of $\mathcal{B}(FD_{2p^m})$ is odd, it follows
that $t_{\alpha} \notin \mathcal{B}(FD_{2p^m})$. Further, if $\alpha = 1 $, then
the element $1 + \widehat{D_{2p^m}}$ can be written as   
 $ 1 + \widehat{D_{2p^m}} = b u_{b, a}u_{b, a^2} \cdots
u_{b, a^\frac{p^m-1}{2}}$. Since $b \notin \mathcal{B}(FD_{2p^m})$, it follows
that $kerT'''=\{ 1 \}$.

 \noindent It is known that \[ SL_2(F(\zeta+\zeta^{-1})) = \Bigg
\langle \left(
 \begin{array}{cc}
   1 & 0 \\
  u & 1 \\
\end{array}
 \right) 
  , \ \ \left(
 \begin{array}{cc}
   1 & v \\
  0 & 1 \\
 \end{array}
 \right) \ \Bigg | \  u, v \in F(\zeta+\zeta^{-1})  \Bigg \rangle .\] For onto
of $T'''$, choose an element $U_{rs}$ from $\displaystyle \mathop
\prod_{r=1}^{m}
\mathop {\prod}_{s=1}^{k_r'}
SL_2(F(\gamma_{rs}+ \gamma_{rs} ^{-1}))$, where the component of $U_{rs}$
corresponding to the matrix representation $T_{rs}$ is 
$\left(
\begin{array}{cc}
  1 & 0 \\
 u_{rs}  & 1 \\
\end{array}
\right),$ such that $u_{rs} \in F(\gamma_{rs}+\gamma_{rs}^{-1})$ 
and the remaining components are identity matrices. Let $y_{rs}= \displaystyle
\prod_{t=1}^{m} \prod_{\stackrel{j=1}{t \neq r \mbox{ and } j \neq
s}}^{k_t'}f'_{t,j}(\gamma_{rs}+\gamma_{rs}^{-1})$, where $f'_{t,j}(x)$ is the
minimal polynomial of $\gamma_{tj} + \gamma_{tj}^{-1}$ over $F$. If we define
$g(x)=\displaystyle \prod_{t=1}^{m} \prod_{\stackrel{j=1}{t
\neq r \mbox{ and } j \neq s}}^{k_t'}f'_{t,j}(x)$,
then $g(\gamma_{tj}+\gamma_{tj}^{-1})=0$ for $1 \leq t \leq m, \ 1 \leq j \leq
k_t', \mbox{ and } {t \neq r \mbox{ and } j \neq s}$ and
$g(\gamma_{rs}+\gamma_{rs}^{-1})= y_{rs}$, a nonzero element of
$F(\gamma_{rs}+\gamma_{rs}^{-1})$. Choose $\{y_{rs},
y_{rs}(\gamma_{rs}+\gamma_{rs}^{-1}), \cdots ,
y_{rs}(\gamma_{rs}+\gamma_{rs}^{-1})^{t_r-1}\}$
as a basis of $F(\gamma_{rs}+\gamma_{rs}^{-1})$ over $F$, where
$t_r=[F(\gamma_{rs}+\gamma_{rs}^{-1}):F]$. Thus, the element $u_{rs}$ can be
written as $u_{rs}= \displaystyle y_{rs}\sum_{q = 0}^{t_r-1}\beta_q
(\gamma_{rs}+\gamma_{rs}^{-1})^{q}$. Assume that $u'(x) = g(x)u(x)$, where
$u(x)=\beta_0+\beta_1x+ \cdots +\beta_{t_r-1} x^{t_r-1}$, a polynomial over $F$.
Now
$u'(\gamma_{rs}+\gamma_{rs}^{-1})=u_{rs}$ and
$u'(\gamma_{tj}+\gamma_{tj}^{-1})=0$ for $1 \leq
t \leq m, 1 \leq j \leq k_t', t \neq r \mbox{ and } j \neq s$. If
$u'(x)= a_0 + a_1 x+ \cdots + a_k x^{k}$, then $U_{rs}$ can be written as
$U_{rs}= e_0e_1 \cdots e_k$, where $e_h$ is an element of
$\prod_{t=1}^{m}\prod_{j=1}^{k_{t}'}SL_2(F(\gamma_{tj}+\gamma_{tj}^{-1}))$ such
that the component ${e_h}^{(t,j)}$ of $e_h$ corresponding to the matrix
representation $T_{tj}$ is $\left(
\begin{array}{cc}
  1 & 0 \\
 a_h(\gamma_{tj} + \gamma_{tj}^{-1})^h & 1 \\
\end{array}
\right)$, for $0 \leq h \leq k$ and for $1 \leq t \leq m, 1 \leq j \leq k_t'$.
Now we will prove that the preimage of $e_h$ is in
$\mathcal{B}(FD_{2p^m})$
under the map $T'''$.

If $a_h=0$, then it is clear that the preimage of $e_h$ is in
$\mathcal{B}(FD_{2p^m})$. If $a_h \neq 0$, then choose $M$ from
$\prod_{t=1}^{m}\prod_{j=1}^{k_{t}'}GL_2(F(\gamma_{tj}+\gamma_{tj}^{-1}))$, such
that the component $M_{tj}$ corresponding to the matrix representation $T_{tj}$
is $M_{tj}= \left(
\begin{array}{cc}
  1 & 1 \\
 \gamma_{tj} & \gamma_{tj}^{-1} \\
\end{array}
\right) $, for $1 \leq t \leq m, 1 \leq j \leq k_t'$.
Note that \[{M_{tj}}^{-1}e_h^{(t,j)}M_{tj}= 
\left(
\begin{array}{cc}
  1 & 0 \\
 0 & 1 \\
\end{array}
\right) + a_h(\gamma_{tj} +
\gamma_{tj}^{-1})^{h-1} \left(
\begin{array}{cc}
  1 & 1 \\
 1 & 1 \\
\end{array}
\right).\]

Now $(\gamma_{tj} + \gamma_{tj}^{-1})^{h-1}= b_0 + \displaystyle \sum_{q
=1}^{k}b_q((\gamma_{tj})^q + (\gamma_{tj}^{-1})^q )$,\\ where $b_i \in F_2$ and 
$k= \begin{cases}
h-1, & \textrm{ if } h > 0 \\
\frac{p^m-1}{2}, & \textrm{ if } h=0 
\end{cases}$

Hence, \[{M_{tj}}^{-1}e_h^{(t,j)}M_{tj} = \left(
\begin{array}{cc}
  1 & 0 \\
 0 & 1 \\
\end{array}
\right) + (b_0' + \sum_{q = 1}^{k} b_q'(\gamma_{tj}^q +
\gamma_{tj}^{-q} ))\left(
\begin{array}{cc}
  1 & 1 \\
 1 & 1 \\
\end{array}
\right),\] where $b_q'=a_hb_q \in F$ and the component of $M^{-1}e_hM$
corresponding to the matrix representation $T_{tj}$ is
${M_{tj}}^{-1}e_h^{(t,j)}M_{tj}$. If \[\displaystyle \alpha = 1 + (b_0' +
\sum_{q = 1}^{k} b_q'(a^q + a^{-q} ))(1+b),\] then $T''(\alpha)= e_h$, 
because $S_{tj}(x)=
M_{tj}^{-1}T_{tj}(x)M_{tj} \ \forall x \in D_{2p^m}$. Now, if $b_0'= 0$, then
\[\displaystyle \alpha = 1 + \sum_{q = 1}^{k} b_q'(a^q + a^{-q})(1+b)
= \displaystyle \prod_{q=1}^{k}(1 + b_q'(a^q + a^{-q})(1+b)),\] 
which is a product of the elements of $B$. Next, if $b_0' \neq 0$, then
consider 
\begin{align*}
T''(\alpha)&= T''(1+b_0'(1+b)+\displaystyle \sum_{q=1}^{k}(a^q+a^{-q})(1+b))\\
           & = T''(1+b_0'(1+b))T''(\displaystyle
\prod_{q=1}^{k}(1+(a^q+a^{-q})(1+b))).
\end{align*}
 Further, since $\widehat{D_{2p^m}} \in ker T'$, therefore
\begin{align*}
T'(1+b_0'(1+b))&= T'((1+b_0'(1+b))(1+\widehat{D_{2p^m}}))\\
               &=T'(1+b_0'(1+b+\widehat{D_{2p^m}}))\\
                &=\prod_{i=1}^{\frac{p^m-1}{2}} T'(1+b_0'(a^i+a^{-i})(1+b)),
\end{align*}                
                is an element of
$T''(\mathcal{B}(FD_{2p^m}))$. Then, $e_h$ has a preimage in
$\mathcal{B}(FD_{2p^m})$. in a similar way, one can prove for other
generators as well. Hence 
\[\mathcal{B}(FD_{2p^m}) \cong
\displaystyle \mathop \prod_{r=1}^{m} \mathop {\prod}_{s=1}^{k_r'}
SL_2(F(\gamma_{rs}+ \gamma_{rs} ^{-1})).\] 

\begin{theo} 
  The unitary subgroup $\mathcal{U}_*(FD_{2p^m})$ of the group algebra
 $FD_{2p^m}$ is the direct product of the subgroup
 $\mathcal{B}(FD_{2p^m})$ with the group $1+F\widehat{D_{2p^m}}$.
Further,
  $\displaystyle \mathcal{U}(FD_{2p^m})=\mathcal{U}_*(FD_{2p^m})
\times \displaystyle \mathop \prod_{r=1}^{m} \mathop {\prod}_{s=1}^{k_r'}
\langle x_{rs} \rangle \times \langle x \rangle$, where
$x$ and $ x_{rs}$  are invertible elements in the 
 center of the group algebra $FD_{2p^m}$ such that the order of $x$ is $2^n-1$
and the order of $2^{nt_r}-1$, if $d$ is odd, otherwise $2^{n(\frac{o_r}{2})}-1$.
 \end{theo}
\begin{proof}
Let $F(\gamma_{rs}+ \gamma_{rs}^{-1})^*= \langle \eta_{rs} \rangle$ be a
cyclic group of order ${2^n}^{t_r}-1$, where $t_r=[F(\gamma_{rs}+
\gamma_{rs}^{-1}):F]$. Then, \[GL_2(F(\gamma_{rs}+ \gamma_{rs} ^{-1}))=
SL_2(F(\gamma_{rs}+ \gamma_{rs} ^{-1}))\times \Bigg\{\left(
\begin{array}{cc}
  \eta_{rs}^j & 0 \\
 0 & \eta_{rs}^j \\
\end{array}
\right) | 1 \leq j \leq {2^n}^{t_r}-1 \Bigg\}.\] Since $\eta_{rs} \in
F(\gamma_{rs}+
\gamma_{rs}^{-1})$ and  $\{y_{rs},
y_{rs}(\gamma_{rs}+\gamma_{rs}^{-1}), \cdots ,
y_{rs}(\gamma_{rs}+\gamma_{rs}^{-1})^{t_r-1}\}$ is a basis of $F(\gamma_{rs}+
\gamma_{rs}^{-1})$, it implies that $\eta_{r,s} =
y_{rs}h_{rs}(\gamma_{rs} + \gamma_{rs}^{-1}) = h_{rs}'(\gamma_{rs}
+\gamma_{rs}^{-1})$, where $h_{rs}(x) \in
F[x] \ \text{and} \ h_{rs}'(x) = g(x)h_{rs}(x)$. Further, note that
$h_{rs}'(\gamma_{tj} + \gamma_{tj}^{-1}) = 0$ for $t \neq r \text{ and} j \neq
s$. If $\alpha_{rs}$ denotes the constant term of the polynomial $h_{rs}'(x)$
over $F$, then the image of $h_{rs}'(a+a^{-1})$ under the map $T'$ is $X_{rs}$,
where the first component of $X_{rs}$ is $\alpha_{rs}$ and the component
corresponding to the matrix representation $T_{rs}$ is $\left(
\begin{array}{cc}
  \eta_{rs} & 0 \\
 0  & \eta_{rs} \\
\end{array}
\right)$ and all the remaining components are zero matrices. Define
$Y(x)=\displaystyle \prod_{t=1}^{m}\prod_{j=1}^{k_t'}f'_{t,j}(x)$, where
$f'_{t,j}(x)$ is the minimal polynomial of $\gamma_{tj} + \gamma_{tj}^{-1}$ over
$F$. Without loss of generality, one can take that the constant coefficient of
$Y(x)$ is 1. Note that $Y(\gamma_{tj} +\gamma_{tj}^{-1}) = 0, \ 1 \leq t \leq
m, 1 \leq j \leq k_t'$. Therefore, the image of $Y(a+a^{-1})$ under the map $T'$
is the element whose first component is 1 and all the remaining components are
zero matrices. Next, we will find a preimage of $Z_{rs}$ where $Z_{rs}$ is the
element of $F^* \times \displaystyle \mathop \prod_{t=1}^{m} \mathop
{\prod}_{j=1}^{k_t'}
GL_2(F(\gamma_{tj}+ \gamma_{tj} ^{-1}))$ such that its first component is 1, the
component corresponding
to the matrix representation $T_{rs}$ is 
$\left(
\begin{array}{cc}
  \eta_{rs} & 0 \\
 0  & \eta_{rs} \\
\end{array}
\right)$ and all the remaining components are identity matrices. Take 
\[z_{rs} =\displaystyle \sum_{t = 1}^{m} \sum_{\stackrel{j=1}{t \neq r \text{
and } j \neq s}}^{k_t'} h_{tj}'(a + a^{-1})^{2^{nt_r} -1}.\]
 The image of
$z_{rs}$ under the map $T'$ is the element whose first component is either 0 or
1 and the component corresponding to the matrix representation $T_{rs}$ is 
$\left(
\begin{array}{cc}
  0 & 0 \\
 0  & 0 \\
\end{array}
\right)$ and all the remaining components are identity matrices. If we define
$z'_{rs}=z_{rs}+h'_{rs}$, then it is clear that the image of $z'_{rs}$ is the
element whose first component is either $\alpha_{rs}$ or $1+\alpha_{rs}$ and the
component corresponding to the matrix representation $T_{rs}$ is 
$\left(
\begin{array}{cc}
  \eta_{rs} & 0 \\
 0  & \eta_{rs} \\
\end{array}
\right)$ and all the remaining components are identity matrices.
Therefore, $z'_{rs}+(1+\alpha_{rs})Y(a+a^{-1})$ or
$z'_{rs}+\alpha_{rs}Y(a+a^{-1})$ is a preimage of the element $Z_{rs}$ which
lies in the center of $FD_{2p^m}$. If $ F^*=\langle \eta \rangle$, then the
element $z = \displaystyle \sum_{t = 1}^{m} \sum_{j=1}^{k_t'} h_{tj}'(a +
a^{-1})^{2^{nt_r} -1}+ \eta Y(a+a^{-1})$ or $z = \displaystyle \sum_{t = 1}^{m}
\sum_{j=1}^{k_t'} h_{tj}'(a + a^{-1})^{2^{nt_r} -1}+ (1+\eta) Y(a+a^{-1})$ is a
preimage of the element $Z \in F^* \times \displaystyle \mathop \prod_{t=1}^{m}
\mathop {\prod}_{j=1}^{k_t'}
GL_2(F(\gamma_{tj}+ \gamma_{tj} ^{-1}))$, whose first component is $\eta$ and
all
the remaining components are identity matrices. 
Therefore, there are elements $x_{rs}$ of order $2^{nt_r}-1$ and an
element $x$ of order $2^n-1$ in $\mathcal{U}(FD_{2p^m})$. Since
$\langle x_{rs} \rangle \cap \displaystyle \langle  
x_{tj} \ | \ 1 \leq t \leq m, 1 \leq s \leq k_t', t \neq r \text{ and } j
\neq s\rangle = \{1\} $, take $W = \displaystyle
\prod_{t = 1}^{m} \prod_{j=1}^{k_t'}\langle
x_{tj} \rangle$. Note that $\langle x \rangle \cap W = \{1 \}$.
The order of the group $ \langle x \rangle \times W$ is odd and is contained in
the center of $FD_{2p^m}$, therefore $W \cap \mathcal{U}_*(FD_{2p^m}) = \{1 \}$.
By compairing the order, we obtain that 
 $\mathcal{U}(FD_{2p^m}) = W \times \langle x \rangle \times
(\mathcal{B}(FD_{2p^m}) \times
(1+F \widehat{D_{2p^m}}))$ and therefore, $\mathcal{U}_{*}(FD_{2p^m})=
\mathcal{B}(FD_{2p^m})
\times
(1+F \widehat{D_{2p^m}})$. The group $1+F \widehat{D_{2p^m}}= \displaystyle
\prod_{i=0}^{n-1}
\langle 1+\alpha^i \widehat{D_{2p^m}} \rangle$, is an elementary abelian
$2$-group.
\end{proof}
 
 \begin{corollary}
 The commutator subgroup
$\mathcal{U}'(FD_{2p^m})=\mathcal{U}_*'(FD_{2p^m})$.
Also, $\mathcal{U}'(FD_{2p^m})$ is a normal subgroup of
$\mathcal{B}(FD_{2p^m})$.
 \end{corollary}
 \begin{proof}
Since $\mathcal{U}(FD_{2p^m})= W \times \langle x \rangle \times
\mathcal{U}_*(FD_{2p^m})$
such that $ W \times \langle x \rangle$
is in the center of $FD_{2p^m}$, it follows that
$\mathcal{U}'(FD_{2p^m}) =
\mathcal{U}_*'(FD_{2p^m})$. 
Further, because $\mathcal{U}_*(FD_{2p^m}) = \mathcal{B}(FD_{2p^m})
\times
(1+F \widehat{D_{2p^m}})$
  and $(1+F \widehat{D_{2p^m}})$ lies in the center of $FD_{2p^m}$, it implies
that 
$\mathcal{U}_*'(FD_{2p^m}) \leq \mathcal{B}(FD_{2p^m}) \leq
 \mathcal{U}_*(FD_{2p^m}) $ providing us with the result.
 \end{proof}

 \bibliographystyle{plain}
 \bibliography{references}

\end{document}